\documentclass{amsart}
\usepackage{amsmath,amsthm,amssymb,amsfonts,graphicx, hyperref}
\usepackage{color}

\title{A copositive formulation for the stability number of infinite graphs}

\author{Cristian Dobre}

\address{C.~Dobre, Biometris, Wageningen University and Research Center, 
6700 HB Wageningen, The Netherlands}
\email{cristian.dobre@wur.nl}

\author{Mirjam D\"{u}r}

\address{M.~D\"{u}r, Department of Mathematics, University of Trier,
  54286 Trier, Germany}
\email{duer@uni-trier.de}

\author{Leonhard Frerick}

\address{L.~Frerick, Department of Mathematics, University of Trier,
  54286 Trier, Germany}
\email{frerick@uni-trier.de}

\author{Frank Vallentin}

\address{F.~Vallentin, Mathematisches Institut, Universit\"at zu
  K\"oln, Weyertal 86--90, 50931 K\"oln, Germany}
\email{frank.vallentin@uni-koeln.de}

\thanks{C.~Dobre and M.~D\"{u}r were supported by Vici grant
  639.033.907 from the Netherlands Organisation for Scientific
  Research (NWO), F.~Vallentin was supported by Vidi grant 639.032.917
  from the Netherlands Organisation for Scientific Research (NWO)}



\date{October 3, 2015}

\newcommand{\defi}[1]{\textit{#1}}

\newcommand{\R}{\mathbb{R}}
\newcommand{\N}{\mathbb{N}}

\newtheorem{defin}{Definition}[section]
\newtheorem{definition}[defin]{Definition}
\newtheorem{proposition}[defin]{Proposition}
\newtheorem{theorem}[defin]{Theorem}

\newtheorem{lemma}[defin]{Lemma}

\DeclareMathOperator{\trace}{trace}

\DeclareMathOperator{\conv}{conv}

\DeclareMathOperator{\cone}{cone}
\DeclareMathOperator{\cl}{cl}
\DeclareMathOperator{\supp}{supp}
\DeclareMathOperator{\sym}{sym}
\DeclareMathOperator{\ex}{ex}

\begin{document}

\begin{abstract}
  In the last decade, copositive formulations have been proposed for a
  variety of combinatorial optimization problems, for example the
  stability number (independence number). In this paper, we generalize
  this approach to infinite graphs and show that the stability number
  of an infinite graph is the optimal solution of some
  infinite-dimensional copositive program. For this we develop a
  duality theory between the primal convex cone of copositive kernels
  and the dual convex cone of completely positive measures. We
  determine the extreme rays of the latter cone, and we illustrate
  this theory with the help of the kissing number problem.
\end{abstract}

\maketitle

\markboth{C.~Dobre, M.~D\"ur, L.~Frerick, F.~Vallentin}{Copositive formulation for the stability number of infinite graphs}

\section{Introduction}

One way to deal with problems in combinatorial optimization is to
(re-)formulate them as convex optimization problems. This is
beneficial because it allows a geometric interpretation of the
original combinatorial problem and because the convexity provides ways
to find bounds or even to certify optimality of solutions.

In the last decade copositive formulations have been proposed for many
$\mathrm{NP}$-hard problems; see the survey of D\"ur \cite{Duer} and
references therein. In these formulations the hardness is entirely
moved into the copositivity constraint. Therefore any progress in
understanding this constraint immediately provides new insights for a
variety of problems.

Bomze, D\"ur, de Klerk, Roos, Quist, and Terlaky
\cite{BomzeDuerDeKlerkQuistRoosTerlaky} were the first to give a
copositive formulation of an $\mathrm{NP}$-hard combinatorial problem,
namely the clique number of a graph. Similarly, de Klerk and Pasechnik
\cite{DeKlerkPasechnik} considered the stability number of a
graph. The \defi{stability number} of a finite, undirected, simple
graph $G = (V, E)$ is
\begin{equation*}
\alpha(G) = \max\{|S| : \text{$S$ is a stable set}\},
\end{equation*}
where $S \subseteq V$ is a \defi{stable set} if for all $x, y \in S$
we have $\{x,y\} \not\in E$. Finding the stability number of a graph
is a fundamental problem in combinatorial optimization and has many
applications. It is one of the most difficult $\mathrm{NP}$-hard
problems, in the sense that even providing an approximation of any
reasonable quality is $\mathrm{NP}$-hard, see
H\aa{}stad~\cite{Hastad}. In \cite[Theorem 2.2]{DeKlerkPasechnik} de
Klerk and Pasechnik gave the following copositive formulation: For
$V = \{1, \ldots, n\}$,
\begin{equation*}
\begin{array}{rll}
\alpha(G) = \min & t\\
&  t \in \R,\; K \in \mathcal{COP}_n,\\
& K(i,i) = t-1 & \text{for all $i \in V$},\\
& K(i,j) = -1 & \text{for all $\{i,j\} \not\in E$,}
\end{array}
\end{equation*}
where $\mathcal{COP}_n$ denotes the convex cone of copositive
$n \times n$-matrices. Recall that a real symmetric matrix
$K \in \R^{n \times n}$ is called \defi{copositive} if
\begin{equation}
\label{eq:copositive definition}
\sum_{i=1}^n \sum_{j=1}^n K(i,j) a_i a_j \geq 0 \; \text{ for all } a_1, \ldots, a_n \geq 0 .
\end{equation}
The space of symmetric matrices is equipped with the usual Frobenius
trace inner product
\begin{equation*}
\langle K, L \rangle = \trace(KL).
\end{equation*}
With this inner product the dual of the copositive cone is the cone of
\defi{completely positive matrices}
\[
\mathcal{CP}_n = \{L \in \mathbb{R}^{n \times n} : \text{$L$
  symmetric, $\langle K, L \rangle \geq 0$ for all $K \in
  \mathcal{COP}_n$}\}.
\]
In \cite[Theorem 3.1 (iii)]{HallNewman1962a} Hall and Newman
determined the extreme rays of this cone. They showed that a matrix
generates an extreme ray of~$\mathcal{CP}_n$ if and only if it is
of the form
\begin{equation*}
 aa^{\sf T} = \left(\sum_{i=1}^n a_ie_i \right) \left( \sum_{i=1}^n a_i e_i \right)^{\sf T} \quad \text{with $a_1, \ldots, a_n \geq 0$,}
\end{equation*}
where $e_1, \ldots, e_n$ is the standard basis of $\R^n$. We write
$aa^{\sf T}$ in this seemingly complicated form because it suggests a
way to generalize this to the infinite setting.

\medskip

The concept of stable sets in graphs is also useful in infinite
graphs, for instance to model geometric packing problems in metric
spaces. Let $V$ be a compact metric space with probability
measure~$\omega$, which is strictly positive on open sets, and
distance function~$d$. Finding the densest packing of balls with
radius~$r$ in~$V$ is equivalent to finding the stability number of the
graph $G = (V, E)$ where the vertices of~$G$ are the elements of the
compact metric space~$V$ and $\{x,y\} \in E$ iff $d(x,y) \in (0,2r)$.
For example, one can formulate the kissing number problem in this way
(see, e.g., the exposition by Pfender and
Ziegler~\cite{PfenderZiegler2004a}): Take $V$ to be the unit sphere
$S^{n-1} = \{x \in \R^n : x^{\sf T} x =1\}$, let $\omega$ be the
normalized, induced Lebesgue measure, $d$ the angular distance
$d(x,y) = \arccos x^{\sf T} y$, and $r = \pi/6$.

\medskip

The graphs defined above are compact topological packing graphs as
introduced by de Laat and Vallentin \cite{Laat}. The formal
definition is as follows.

\begin{definition}
\label{def:topological packing graph}
A graph whose vertex set is a Hausdorff topological space is called a
\emph{topological packing graph} if each finite clique is contained in
an open clique. An \emph{open clique} is an open subset of the vertex
set where every two vertices are adjacent.
\end{definition}

In the remainder of the paper we assume that $G = (V, E)$ is a compact
topological packing graph where the vertex set $V$ is
metrizable. Compactness implies that the stability number of $G$ is
finite.

\medskip

The aim of the present paper is to give a copositive formulation of
the stability number of compact topological packing graphs, and
thereby to initiate the study of copositive formulations also for
other combinatorial problems in a continuous setting.

\medskip

In order to go from the finite to the infinite setting we have to
provide the right infinite dimensional generalizations of copositive
and completely positive matrices. For this we will apply classical
notions and results from infinite dimensional convexity theory, see
the books by Barvinok~\cite{Barvinok}, by Simon~\cite{Simon}, or Rudin~\cite{Rudin}.

On the copositive side, the natural generalization of finite
$n \times n$-matrices (where rows and columns are indexed by
$\{1, \ldots, n\}$) are real-valued continuous \defi{Hilbert-Schmidt
  kernels} on the compact Hausdorff space $V$. The set of continuous
Hilbert-Schmidt kernels is defined as follows:
\begin{equation*}
C(V \times V)_{\sym} = \{K \colon V \times V \to \R : \text{$K$ is symmetric and continuous}\}.
\end{equation*}
Symmetry means here that for all $x, y \in V$ we have $K(x,y) =
K(y,x)$.

In analogy to (\ref{eq:copositive definition}) we call a kernel $K \in
C(V \times V)_{\sym}$ \defi{copositive} if
\begin{equation}
\label{eq:copositive integral definition}
\int_V \int_V K(x,y) f(x) f(y) \, d\omega(x) d\omega(y) \geq 0 \;
\text{ for all } f \in C(V)_{\geq 0},
\end{equation}
where $C(V)_{\geq 0}$ is the convex cone of all nonnegative continuous
functions on~$V$. We denote the convex cone of copositive kernels
by $\mathcal{COP}_V$.

Note that the set $\mathcal{COP}_V$ is independent of the choice of
$\omega$. This follows from Lemma~\ref{le:characterizationCopositive}
below. In the finite dimensional case, this corresponds to the fact
that $K \in \mathcal{COP}_n$ if and only if $DKD \in \mathcal{COP}_n$
for any diagonal matrix $D$ with strictly positive diagonal
elements. In the finite setting, this scaling invariance is relevant
when studying approximations of $\mathcal{COP}_n$,
cf.~\cite{DickinsonEtAl2013}.

Also note that a measure $\omega$ with the required property (strict
positivity on open sets) always exists: Let $V$ be compact and
metrizable. Then $V$ is separable, so there exists a sequence
$(x_n)_{n\in\N}$ which is dense in $V$. Choose a sequence $a_n > 0$
with $\sum_{n\in\N} a_n = 1$ and for any set $A\subset V$ define
$\omega(A) := \sum_{n\in\N} a_n \chi_A(x_n)$. Then $\omega$ is a
probability measure with the required property and consequently the
set $\mathcal{COP}_V$ is well-defined.

The dual space of $C(V \times V)_{\sym}$ equipped with the
supremum norm consists of all continuous linear functionals. By Riesz'
representation theorem it can be identified with the space of
symmetric, signed Radon measures $M(V \times V)_{\sym}$ equipped
with the total variation norm (see e.g.\ \cite[Chapter
2.2]{BergChristensenRessel1984}). A \emph{signed Radon measure} is the
difference of two Radon measures, where a \emph{Radon measure} $\mu$
is a locally finite measure on the Borel algebra satisfying
\emph{inner regularity}:
$\mu(B) = \sup \{ \mu(C) : C \subseteq B, \, C \text{ compact}\}$ for
each Borel set $B$. A Radon measure $\mu$ is \emph{symmetric} if
$\mu(A \times B) = \mu(B \times A)$ for all measurable
$A, B \subseteq V$.

Let $K$ be a continuous kernel and $\mu$ be a symmetric, signed Radon
measure on $V \times V$. The duality is given by the pairing
\begin{equation}
\label{eq:inner product}
\langle K, \mu \rangle = \int_{V \times V} K(x,y) \, d\mu(x,y).
\end{equation}
We endow the spaces with the weakest topologies compatible with the
pairing: the weak topology on $C(V \times V)_{\sym}$  and the weak*
topology on $M(V \times V)_{\sym}$. Then the dual cone of the cone $\mathcal{COP}_V$ of copositive kernels is a cone 
which we call the \defi{cone of completely positive measures},
\begin{equation*}
\mathcal{CP}_V = \{\mu \in M(V \times V)_{\sym} : \langle K, \mu \rangle
\geq 0 \; \text{for all $K \in \mathcal{COP}_V$}\}.
\end{equation*}

Using these definitions we can state our main theorem which gives a
copositive formulation of the stability number of compact topological
packing graphs.

\begin{theorem}
\label{th:main}
Let $G = (V,E)$ be a compact topological packing graph. Then the
stability number of $G$ equals
\begin{equation*}
\label{eq:copositive}
\tag{P}
\begin{array}{rll}
\alpha(G) = \inf & t\\
&  t \in \R,\; K \in \mathcal{COP}_V \\
& K(x,x) = t-1 & \text{for all $x \in V$}\\
& K(x,y) = -1 & \text{for all $\{x,y\} \not\in E$.}
\end{array}
\end{equation*}
\end{theorem}

The remainder of the paper is organized as follows: In
Section~\ref{sec:cones} we analyze properties of the two
infinite-dimensional cones $\mathcal{COP}_V$ and $\mathcal{CP}_V$; we
give a characterization of copositive kernels and we determine the
extreme rays of the cone $\mathcal{CP}_V$ of completely positive
measures. In Section~\ref{sec:formulation} we prove our main result,
Theorem~\ref{th:main}. There we first derive a completely positive
formulation of the stability number --- which we will denote
by~\eqref{eq:completelypositive} ---, that is the dual
of~\eqref{eq:copositive}. Then by proving that there is no duality gap
between the primal and the dual we derive Theorem~\ref{th:main}.  We
also give a version of Theorem~\ref{th:main} for the the weighted
stability number. In Section~\ref{sec:kissing} we provide an
interpretation of our copositive formulation for the kissing number
problem. Then we end by posing a question for possible future work.

\section{Copositive kernels and completely positive measures}
\label{sec:cones}

\subsection{Copositive kernels}

In \eqref{eq:copositive integral definition} we defined a kernel to be
copositive by integrating it with nonnegative continuous functions.
Instead of using nonnegative continuous functions we can also define
copositivity by means of finite nonnegative delta measures.  For the
larger class of positive definite kernels this is a classical fact
which holds under the same assumptions on $V$ and $\omega$ as imposed
here, as realized for instance by Bochner \cite[Lemma 1]{Bochner},
see also Folland \cite[Proposition 3.35]{Folland1995a}:

A kernel $K \in C(V \times V)_{\sym}$ is called \emph{positive (semi-)definite}  if
\begin{equation*}
\int_V \int_V K(x,y) f(x) f(y) \, d\omega(x) d\omega(y) \geq 0 \;
\text{ for all } f \in C(V),
\end{equation*}
where $C(V)$ denotes the space of continuous functions. Bochner~\cite{Bochner}
proved that a kernel $K \in C(V \times V)_{\sym}$ is positive semidefinite
if and only if for any choice $x_{1}, \ldots, x_{N}$ of finitely many
points in $V$, the matrix $(K(x_{i},x_{j}))_{i,j=1}^{N}$ is positive
semidefinite.

The following lemma shows that a similar characterization holds
for copositive kernels. For the reader's convenience we provide a proof here.

\begin{lemma}
\label{le:characterizationCopositive}
Let $V$ be a compact space with probability measure $\omega$ which is
strictly positive on open sets. A kernel $K \in C(V \times V)_{\sym}$ is copositive if and only if for any choice of 
finitely many points $x_{1},\ldots, x_{N} \in V$,
the matrix $(K(x_{i},x_{j}))_{i,j=1}^{N}$ is copositive.
\end{lemma}

\begin{proof}
Since $V \times V$ is compact, the continuous function $K$ is
uniformly continuous and bounded on $V \times V$.

Suppose that for any choice $x_{1},\ldots, x_{N}$ of finitely many 
points in $V$, the matrix $(K(x_{i},x_{j}))_{i,j=1}^{N}$ is copositive. 
Let $\varepsilon > 0$ and $f \in C(V)_{\geq 0}$. Since $K$ is 
uniformly continuous, we can partition~$V$ into a finite number of 
measurable sets $V_{1},\ldots,V_{N}$ and find points $x_i \in V_i$ such that
\begin{equation}
\label{eq:partition}
|K(x,y) - K(x_i, x_j)| \leq \varepsilon \quad \text{for all $x \in
  V_i$, $y \in V_j$}.
\end{equation}
Set $a_{i}=\int_{V_i} f(x) \, d\omega(x)$. Then $a_i \ge 0$ and
\begin{align*}
& \left|
    \sum_{i=1}^{N}\sum_{j=1}^{N}K(x_{i},x_{j})a_{i}a_{j}-\int_{V}\int_{V}K(x,y)f(x)f(y)
    \, d\omega(x)d\omega(y) \right|  \\
&= \left| \sum_{i=1}^{N}\sum_{j=1}^{N}
  \int_{V_{i}}\int_{V_{j}}(K(x_{i},x_{j})-K(x,y))f(x)f(y) \, d\omega(x)d\omega(y)  \right|  \\
&\leq  \sum_{i=1}^{N}\sum_{j=1}^{N} \int_{V_{i}}\int_{V_{j}} \left|
  K(x_{i},x_{j})-K(x,y) \right| f(x)f(y) \, d\omega(x)d\omega(y) \\
&\leq  \varepsilon \int_{V}\int_{V}f(x)f(y) \, d\omega(x)d\omega(y)
  \to 0
\end{align*}
as $\varepsilon \to 0$. One obtains 
\[
\int_{V}\int_{V} K(x,y)f(x)f(y) \, d\omega(x)d\omega(y) \geq 0,
\]
and hence $K$ is copositive.

Conversely, assume $K$ is copositive. Then, inequality
\eqref{eq:copositive integral definition} holds also for the larger
class of integrable functions $f \in L^1(V)$ which are nonnegative
$\omega$-almost everywhere. This follows since the family of
nonnegative continuous functions lies dense in the family of
$\omega$-almost everywhere nonnegative $L^1$-functions. Let
$x_{1},\ldots,x_{N} \in V$ (we may assume the $x_i$'s are pairwise
different), let $a_{1},\ldots,a_{N} \geq 0$ and let
$\varepsilon > 0$. We construct disjoint open neighborhoods $V_i$ of
$x_{i}$ such that~\eqref{eq:partition} holds. By assumption, since
$\omega$ is strictly positive on open sets, we have $\omega(V_i) > 0$.
Thus the integrable function
\begin{equation*}
f(x)= \left \{
\begin{array} {cl}
\frac{a_i}{\omega(V_i)} & \text{if $x \in V_{i}$,} \\
0 & \text{otherwise,}
\end{array}
\right.
\end{equation*}
is nonnegative on $V$, and $K(x_{i},x_{j}) a_i a_j$ can be expressed as
\begin{equation*}
 K(x_i,x_j) a_i a_j = \int_{V_i}\int_{V_j} K(x_i, x_j) f(x) f(y) \, d\omega(x)d\omega(y).
\end{equation*}
Then
\[
\left|
  \sum_{i=1}^{N}\sum_{j=1}^{N}K(x_{i},x_{j})a_{i}a_{j}-\int_{V}\int_{V}K(x,y)f(x)f(y)
  \, d\omega(x)d\omega(y) \right|
\leq  \varepsilon \sum_{i=1}^{N} \sum_{j=1}^{N} a_{i}a_{j}.
\]
By letting $\varepsilon$ tend to zero, one obtains $\sum_{i=1}^{N}
\sum_{j=1}^{N} K(x_i,x_j) a_i a_j \geq 0$, which concludes the proof.
\end{proof}

An alternative characterization of copositive kernels was first noted
by Pfender \cite[Lemma 3.3]{Pfender} in the case when $V$ is the unit
sphere $S^{n-1}$ and for copositive kernels which are invariant under
the orthogonal group. In fact, this holds in general.

\begin{lemma}
\label{lem:pfender}
A kernel $K \in C(V \times V)_{\sym}$ 
is copositive if and only if for any choice of finitely many
points $x_1,\ldots, x_N$ in $V$, the sum $\sum_{i=1}^N \sum_{j =1}^N
K(x_i,x_j)$ is nonnegative.
\end{lemma}

\begin{proof}
  Assume that $K$ is copositive and take any finite set of points
  $x_1,\ldots, x_N \in V$.  Then choosing all $a_i$'s in the previous
  lemma equal to one gives 
\[
\sum_{i=1}^N \sum_{j =1}^N K(x_i,x_j) \ge 0.
\]

To show the converse, let $f \in C(V)_{\geq 0}$. By scaling we may
assume that $f$ is a probability density function on~$V$, i.e.\
$\int_V f(x) \, d\omega(x) = 1$.  Picking points
$x_1, \ldots, x_N \in V$ independently at random from this
distribution gives
\begin{align*}
0 &\le \mathbb{E} \left[\frac{1}{N^2} \sum_{i=1}^N \sum_{j=1}^N K(x_i,x_j)\right]\\
&= \frac{1}{N} \int_V K(x,x) f(x) \, d\omega(x) +
\frac{N-1}{N} \int_V \int_V K(x,y) f(x) f(y) \, d\omega(x) d\omega(y).
\end{align*}
By letting $N$ tend to infinity, we see that the double integral is
nonnegative, and hence $K$ is copositive.
\end{proof}

\subsection{Completely positive measures}

In the finite setting, the rank-$1$-matrices
\[
aa^{\sf T} = \left(\sum_{i=1}^n a_i e_i\right)\left(\sum_{i=1}^n a_i
  e_i\right)^{\sf T}
\quad \text{with} \quad
a_1, \ldots, a_n \geq 0,
\]
where $e_1, \ldots, e_n$ is the standard basis of $\R^n$, determine
all extreme rays of the cone of completely positive
matrices~$\mathcal{CP}_n$. So we have an explicit description of this
cone,
\begin{equation*}
\mathcal{CP}_n = \cone\left\{aa^{\sf T} : a_1, \ldots, a_n \geq 0\right\}.
\end{equation*}
In a sense, this fact generalizes to the infinite setting as we shall
show soon in Theorem~\ref{th:extreme rays}.

We have to find the proper replacement of the rank-$1$-matrices
$aa^{\sf T}$.  We will show next that delta measures of the form
\begin{equation}
\label{eq:delta}
\sum_{i=1}^N a_i \delta_{x_i} \otimes \sum_{i=1}^N a_i
\delta_{x_i} \quad \text{with}  \quad x_1, \ldots, x_N \in V, \; a_1,
\ldots, a_N \geq 0
\end{equation}
defined by
\[
\left\langle K, \sum_{i=1}^N a_i
\delta_{x_i} \otimes \sum_{i=1}^N a_i \delta_{x_i}
\right\rangle = \sum_{i=1}^N\sum_{j=1}^N a_i a_j K(x_i, x_j) \quad \text{ for  $K \in C(V \times V)_{\sym}$}
\]
play a similar role as the rank-$1$-matrices.

\begin{proposition}
\label{prop:closure}
The cone of completely positive measures equals
\[
\mathcal{CP}_V = \cl \cone\left\{\sum_{i=1}^N a_i \delta_{x_i} \otimes
\sum_{i=1}^N a_i \delta_{x_i} : N \in \mathbb{N}, \; x_i \in V, \;
a_i \geq 0\right\},
\]
where the closure is taken with respect to the weak* topology.
\end{proposition}

\begin{proof}
  One inclusion is straightforward: By definition, the
  cone~$\mathcal{CP}_V$ of completely positive measures is closed, and
  by Lemma~\ref{le:characterizationCopositive} delta measures of the
  form~\eqref{eq:delta} lie in~$\mathcal{CP}_V$.

  For the other inclusion we use the Hahn-Banach theorem for locally
  convex topological vector spaces. For this note that
  $M(V \times V)_{\sym}$ with the weak* topology is a locally convex
  topological space, and all continuous linear functionals are given
  by $\langle K, \cdot \rangle$ for some $K \in C(V \times
  V)_{\sym}$. Take
\[ 
\mu \in M(V \times V)_{\sym} \setminus \cl \cone\left\{\sum_{i=1}^N
  a_i \delta_{x_i} \otimes \sum_{i=1}^N a_i \delta_{x_i} : N \in
  \mathbb{N}, \; x_i \in V, \; a_i \geq 0\right\}.
\]
By Hahn-Banach there exists a kernel $K \in C(V \times V)_{\sym}$ such
that $\langle K, \mu \rangle < 0$ and
\[
\left\langle K, \sum_{i=1}^N a_i \delta_{x_i} \otimes
  \sum_{i=1}^N a_i \delta_{x_i} \right\rangle 
  \geq 0 \quad \text{ for all $N \in \N, x_i \in V$ and $a_i \ge 0$ ($i=1,\ldots,N$).} 
\]
Hence, again by Lemma~\ref{le:characterizationCopositive}, the
kernel~$K$ is copositive and therefore $\mu \not\in \mathcal{CP}_V$.
\end{proof}

However, it turns out that we really need to take the closure in the
statement of Proposition~\ref{prop:closure}. In particular, the set of
extreme rays of the cone of completely positive measures is strictly
larger than the set of delta measures given in the proposition. The
set of extreme rays consists of all product measures of
$\mu \times \mu$ where $\mu$ is a nonnegative measure on $V$:

\begin{theorem}
\label{th:extreme rays}
A measure generates an extreme ray of the cone $\mathcal{CP}_V$ of
completely positive measures if and only if it is a product measure of
the form $\mu \otimes \mu$, where $\mu \in M(V)$ is a nonnegative
measure on $V$.
\end{theorem}

Before proving the theorem, we first want to describe our strategy. We
start by cutting $\mathcal{CP}_V$ into compact convex slices
$\lambda \mathcal{B}$ where $\mathcal{B}$ is the closure of the convex
hull of all product measures of finitely supported probability
measures (Lemma~\ref{l1} and Proposition~\ref{prop:compact}). Then we
consider in the proof of Theorem~\ref{th:extreme rays} set
$\mathcal{K}_1$ which consists of all product measures of all
probability measures on~$V$. It is clear that $\mathcal{B}$ equals
$\cl \conv \mathcal{K}_1$. From Milman's converse of the Krein-Milman
theorem (Theorem~\ref{the:Milman}) we get immediately that extreme
points of $\mathcal{B}$ are contained in $\mathcal{K}_1$. Proving the
converse inclusion requires work. For this we rely on Choquet's
theorem (Theorem~\ref{the:Choquet}).

\medskip

The following general lemma is a slight variation of
\cite[Lemma~III.2.10]{Barvinok} where we do not use the convexity
assumption.

\begin{lemma}\label{l1}
  Let $\mathcal{B}$ be a compact set in a topological vector space such that
  $0 \not \in \mathcal{B}$. Then the set $\mathcal{K}$ defined by the union
  $\mathcal{K} = \bigcup_{\lambda \geq 0} \lambda \mathcal{B}$ is closed.
\end{lemma}

\begin{proof} 
  We shall show that the complement of $\mathcal{K}$ is open. Let
  $u\not\in \mathcal{K}$.  Since $0\not\in \mathcal{B}$, there is a
  neighborhood $W$ of $0$ that does not intersect $\mathcal{B}$. Let
  $U_1$ be a neighborhood of $u$, and $\delta>0$ such that
  $\alpha U_1 \subset W$ for all $|\alpha|<\delta$ (from the
  continuity of $(\alpha,x)\to \alpha x$ at $(0,u)$). Then
  $U_1\cap \lambda \mathcal{B} = \emptyset$ for all
  $\lambda>1/\delta$.  The image of the compact set
  $[0,1/\delta]\times \mathcal{B}$ by the continuous map
  $(\alpha,x)\to \alpha x$ is compact and is contained in
  $\mathcal{K}$. Hence there is a neighborhood $U_2$ of $u$ that does
  not intersect the image. Then the intersection $U_1\cap U_2$ is a
  neighborhood of $u$ that does not intersect $\mathcal{K}$ which
  proves that $\mathcal{K}$ is closed.
\end{proof}

\begin{proposition}
\label{prop:compact}
The set
\[
\mathcal{B} = \cl\conv\left\{\sum_{i=1}^N a_i \delta_{x_i} \otimes
  \sum_{i=1}^N a_i \delta_{x_i} : N \in \mathbb{N}, \; x_i \in V, \;
  a_i \geq 0,\;  \sum_{i=1}^N a_i = 1\right \}
\]
is weak* compact and equality
\[
\mathcal{CP}_V = \bigcup_{\lambda \geq 0} \lambda \mathcal{B}.
\]
holds. Hence the extreme rays of $\mathcal{CP}_V$ are precisely the
rays generated by the extreme points of $\mathcal{B}$.
\end{proposition}

\begin{proof}
  The set~$\mathcal{B}$ is closed by definition, so in order to prove
  the weak* compactness it suffices to show that~$\mathcal{B}$ is
  contained in a compact set. But this is clear since
\[
\mathcal{B} \subseteq \{\mu \in M(V \times V)_{\sym} : \mu(V \times V) \leq 1, \; \mu
\geq 0\}
\]
and the latter set is compact in the weak* topology by the Theorem of
Banach-Alaoglu.

Since $\mathcal{B}$ does not contain the origin, the union
$\bigcup_{\lambda \geq 0} \lambda \mathcal{B}$ is closed by
Lemma~\ref{l1} and so the desired equality follows by
Proposition~\ref{prop:closure}.
\end{proof}

We cite two results from Choquet theory, see Phelps~\cite{Phelps}.

\begin{theorem}[Milman's converse of the Krein-Milman
  theorem] \label{the:Milman} Suppose that $X$ is a compact convex
  subset of a locally convex space. Suppose further that
  $Z \subseteq X$, and that $X = \cl \conv Z$. Then the extreme points
  of $X$ are contained in the closure of $Z$, i.e.\
  $\ex X \subseteq \cl Z$.
\end{theorem}

\begin{theorem}[Choquet] \label{the:Choquet} Suppose that $X$ is a
  metrizable compact convex subset of a locally convex space $E$, and
  let $x_0 \in X$. Then there exists a probability measure $P$ on~$X$
  which represents $x_0$, i.e.,
\[ 
u(x_0) = \int_X u(x) \, dP(x) \quad \text{ for every continuous linear
  functional $u$ on } E, 
\]
and is supported by the extreme points of~$X$, i.e.,
$P \big( X \setminus \ex X \big) = 0$.
\end{theorem}

Now we are ready to prove the main result of this section.

\begin {proof}[Proof of Theorem~\ref{th:extreme rays}]
We define the following two sets:
\[
\mathcal{M}^+_1(V) = \{\mu \in M(V) : \mu \geq 0, \mu(V)=1\},
\quad
\mathcal{K}_1 = \{\mu \otimes \mu : \mu \in \mathcal{M}^+_1(V)\}.
\]
We will show that
\[
\ex \cl \conv \mathcal{K}_1 = \mathcal{K}_1.
\]
The set $\mathcal{K}_1$ is weak* compact. Therefore, Milman's theorem
(Theorem~\ref{the:Milman}) gives the first inclusion
\[ 
\ex \cl \conv \mathcal{K}_1 \subseteq \mathcal{K}_1. 
\]
To show the converse, assume that $\mu \otimes \mu \in \mathcal{K}_1$
can be written as $\mu \otimes \mu = \tfrac{1}{2}(\nu_1+\nu_2)$ for
some $\nu_1, \nu_2 \in \cl \conv \mathcal{K}_1$.  Since
$\mathcal{K}_1$ is weak* compact and weak* metrizable, it follows from
Choquet's theorem (Theorem~\ref{the:Choquet}) that there exist
probability measures $P_1, P_2$ on $\mathcal{M}^+_1$ such that for all
$u \in \big(M(V\times V)_{\sym}\big)^*$ we have the representation
\begin{equation} \label{Eq:choqrepres}
  u(\nu_i) = \int_{\mathcal{M}^+_1} u(\rho \otimes \rho) \, dP_i(\rho),
 \quad i=1,2.
\end{equation}  
Setting $P := \tfrac{1}{2}P_1 + \tfrac{1}{2}P_2$, we conclude that for all $F \in C(V\times V)_{\sym}$,
\[ (\mu \otimes \mu)(F) = \int_{\mathcal{M}^+_1} (\rho \otimes
\rho)(F) \, dP(\rho).
\]
Since $V$ is a compact metrizable space, the space $C(V)$ of
continuous functions on $V$ is separable. Therefore, there exists a
countable dense subset $H$ of $C(V)_{\geq 0}$.

Take $f \in H$, let $\mathbf{1}_V$ be the constant function equal to $1$ on $V$, and consider
\[ F := \tfrac{1}{2}(f \otimes \mathbf{1}_V ) + \tfrac{1}{2}(\mathbf{1}_V \otimes f ). \] 
Then
\begin{equation} \label{Eq:mu}
   \mu(f) = (\mu \otimes \mu) (F) 
    = \int_{\mathcal{M}^+_1} (\rho \otimes \rho)(F) \, dP(\rho) 
=   \int_{\mathcal{M}^+_1} \rho(f) \, dP(\rho).
\end{equation}
Similarly, consider $F' = f \otimes f$ to obtain
\begin{equation} \label{Eq:musquare}
   \mu(f)^2 = (\mu \otimes \mu) (F') 
    = \int_{\mathcal{M}^+_1}(\rho \otimes \rho)(F') \, dP(\rho) 
    = \int_{\mathcal{M}^+_1} \rho(f)^2 \, dP(\rho).
\end{equation}
Now if $\mu(f)=0$, then~(\ref{Eq:mu}) gives that $\rho(f)=0$ $P$-almost everywhere.
If $\mu(f) > 0$, then combining~(\ref{Eq:mu}) and~(\ref{Eq:musquare}) gives
\[  \int_{\mathcal{M}^+_1} \frac{\rho(f)}{\mu(f)} \, dP(\rho) = 1
     = \int_{\mathcal{M}^+_1} \frac{\rho(f)^2}{\mu(f)^2} \, dP(\rho),
\]
which implies that there exists a set $N_f \subset \mathcal{M}^+_1$ with
$P(N_f) = 0$ such that $\rho (f) = \mu(f)$ for all $\rho \in
\mathcal{M}^+_1 \setminus N_f$. Set $N = \bigcup_{f \in H} N_f$ and
since $H$ is countable, we have $P(N) = 0$ and
\[ 
\rho (f) = \mu(f) \quad \text{for all } \rho\in \mathcal{M}^+_1 \setminus N \text{ and for all } f \in H. 
\]
As $H$ is dense in $C(V)_{\geq 0}$, we get $\rho = \mu$ for all $\rho  \in \mathcal{M}^+_1 \setminus N$.

Since $0 \le P_i(N) \le 2 P(N) = 0$, we obtain that for $i=1,2$ and
for all $F \in C(V \times V)_{\sym}$,
\begin{align*}
 \nu_i(F) &
   = \int_{\mathcal{M}^+_1 \setminus N} (\rho \otimes \rho)(F) \, dP_i(\rho) \\
   & = \int_{\mathcal{M}^+_1 \setminus N} (\mu \otimes \mu)(F) \, dP_i(\rho) \\
   & = (\mu \otimes \mu)(F) \int_{\mathcal{M}^+_1 \setminus N} dP_i(\rho) \\[.5ex]
   & = (\mu \otimes \mu)(F). 
\end{align*}
Hence, $\nu_1 = \nu_2 = \mu \otimes \mu$, which means that $\mu
\otimes \mu \in \ex \cl \conv \mathcal{K}_1 $, and the proof of the
converse inclusion is complete.

The theorem now follows from $\cl \conv \mathcal{K}_1 = \mathcal{B}$
and Proposition \ref{prop:compact}.
\end{proof}

\section{Copositive formulation for the stability number of infinite graphs}
\label{sec:formulation}

In order to develop our copositive formulation of the stability number we make
use of Kantorovich's approach to linear programming over cones in the
framework of locally convex topological vector spaces. This theory is
thoroughly explained in Barvinok~\cite[Chapter IV]{Barvinok} and we
follow his notation closely.

In Section~\ref{ssec:primaldual} we cast the copositive
problem~\eqref{eq:copositive} into the general framework of conic
problems as studied by Barvinok, and using this general theory, we
derive the dual of~\eqref{eq:copositive} which will turn out to be an
infinite-dimensional completely positive problem.  Then we prove our
main theorem, Theorem~\ref{th:main}, in two steps. In the first step,
we show in Section~\ref{ssec:completelypositive} that the stability
number of~$G$ equals the optimal value of the dual problem. In
particular we show that the optimum is attained. In the second step,
Section~\ref{ssec:copositive}, we establish the fact that there is no
duality gap between primal and dual. In Section~\ref{ssec:weighted} we
extend these results and give a copositive formulation for the
weighted stability number.

\subsection{Primal-dual pair}
\label{ssec:primaldual}

As before, let $G = (V,E)$ be a compact topological packing graph with
metrizable vertex set. 
For this graph, the copositive problem~\eqref{eq:copositive} can be seen as a 
general conic problem of the form
\begin{equation}
\label{eq:primal}
\begin{array}{rll}
\inf & \langle x, c \rangle_1\\
&  x \in \mathcal{K}, \; Ax = b
\end{array}
\end{equation}
with the following notations:
\begin{itemize}
\item[] $x = (t, K) \in \R \times C(V \times V)_{\sym}$\\[-1.5ex]
\item[] $c = (1,0) \in \R \times M(V \times V)_{\sym}$\\[-1.5ex]
\item[] $\langle \cdot, \cdot \rangle_1 : (\R \times C(V \times V)_{\sym})
  \times (\R \times M(V \times V)_{\sym}) \to \R$\\[-1.5ex]
\item[] $\mathcal{K} = \R_{\geq 0} \times \mathcal{COP}_V$\\[-1.5ex]
\item[] $A : \R \times C(V \times V)_{\sym} \to C(V) \times C(\overline{E})$
\item[] $\quad A(t,K) = (x \mapsto K(x,x) - t, (x,y) \mapsto K(x,y))$\\[-1.5ex]
\item[] $b = (-1,-1) \in C(V) \times C(\overline{E})$.
\end{itemize}
Here $\overline{E} = \{\{x,y\} : x \neq
y, \{x,y\} \not\in E\}$ is the complement of the edge set. 
Note that we can replace the constraint $t \in \R$ in~\eqref{eq:copositive} by
$t \in \R_{\geq 0}$, since $t \ge 1$ holds automatically
because diagonal elements of copositive kernels are nonnegative.

The dual problem of~\eqref{eq:primal} is
\begin{equation}
\label{eq:dual}
\begin{array}{rll}
\sup & \langle b, y \rangle_2 \\
& c - A^*y \in \mathcal{K}^*.
\end{array}
\end{equation}
Applying this to our setting, it is not difficult to see that we need
\begin{itemize}
\item[] $\langle \cdot, \cdot \rangle_2 : (C(V) \times C(\overline{E})) \times (M(V) \times M(\overline{E})) \to \R$\\[-1.5ex]
\item[] $y = (\mu_0, \mu_1) \in M(V) \times M(\overline{E})$\\[-1.5ex]
\item[] $A^* : M(V) \times M(\overline{E}) \to \R \times M(V \times V)_{\sym}$
\item[] $\quad A^*(\mu_0, \mu_1) = (-\mu_0(V), \mu_0 + \mu_1)$\\[-1.5ex]
\item[] $\mathcal{K}^* = \R_{\geq 0} \times \mathcal{CP}_V$.
\end{itemize}
The map $A^*$ is the adjoint of $A$ because
\begin{equation*}
\begin{split}
\langle A(t,K), (\mu_0,\mu_1) \rangle_2 & =  \int_V K(x,x) - t \, d\mu_0(x) +
\int_{\overline{E}} K(x,y) \, d\mu_1(x,y)\\
& =  -t\mu_0(V) + \int_{V \times V} K(x,y) \, d(\mu_0 + \mu_1)(x,y)\\
& = \langle (t,K), A^*(\mu_0, \mu_1) \rangle_1.
\end{split}
\end{equation*}
Above, when we add the measures~$\mu_0$ and~$\mu_1$, we consider them
as measures defined on the product space~$V \times V$, where we see
the measure~$\mu_0$ as a measure defined on the diagonal~$D = \{(x,x)
: x \in V\}$.

With this, the dual of~\eqref{eq:copositive} is the completely positive program
\begin{equation*}
\begin{array}{rll}
\sup & -\mu_0(D) -\mu_1(\overline{E}) \\
& \mu_0 \in M(D), \; \mu_1 \in M(\overline{E})\\
& 1 + \mu_0(D) \geq 0 \\
& -\mu_0 - \mu_1 \in \mathcal{CP}_V.\\
\end{array}
\end{equation*}
To simplify this dual, we define the \defi{support} of a measure~$\mu$
as follows:
\[
  \supp \mu = 
(V \times V) \setminus O,
\]
where $O$ is the inclusionwise largest open set with $\mu(O) = 0$.
Note that $O$ is given by
\[ 
O = \bigcup_{\small \begin{array}{c} W \text{ open in } V \times V \\
                      \mu(W) = 0 \end{array}} W. 
\]
Then the \defi{dual, completely positive program} equals
\begin{equation*}
\begin{array}{rll}
\sup & \mu(V \times V)\\
& \mu \in \mathcal{CP}_V \\
& \mu(D) \leq 1 \\
& \supp \mu \subseteq D \cup \overline{E}.
\end{array}
\end{equation*}
One can argue by scaling the inequality constraint $\mu(D)
\leq 1$ can be replaced by the equality
constraint $\mu(D) = 1$ and therefore we get
\begin{equation*}
\tag{D}
\label{eq:completelypositive}
\begin{array}{rll}
\sup & \mu(V \times V)\\
& \mu \in \mathcal{CP}_V \\
& \mu(D) = 1 \\
& \supp \mu \subseteq D \cup \overline{E}.
\end{array}
\end{equation*}

This completely positive program using measures is a generalization of
the finite-dimensional completely positive program for finite
graphs~$G = (V,E)$, with $V = \{1, \ldots, n\}$, of de Klerk,
Pasechnik \cite{DeKlerkPasechnik}:
\begin{equation*}
\begin{array}{rll}
\max & \sum\limits_{i=1}^n \sum\limits_{j=1}^n X(i,j) \\[.5ex]
& X \in \mathcal{CP}_n \\
& \sum_{i=1}^n X(i,i) = 1 \\
& X(i,j) = 0 & \text{for all $\{i,j\} \in E$.}
\end{array}
\end{equation*}

\subsection{Completely positive formulation}
\label{ssec:completelypositive}

We next show that the optimal value of
problem~\eqref{eq:completelypositive} equals the stability number.

\begin{theorem}
\label{th:completelypositive}
  Let $G = (V,E)$ be a compact topological packing graph. Then the optimal value of the
  completely positive program~\eqref{eq:completelypositive} is attained and
  equals~$\alpha(G)$.
\end{theorem}

\begin{proof}
Let $\lambda$ be the optimal value of~\eqref{eq:completelypositive}.
For the ease of notation we write $\alpha$ for $\alpha(G)$ in this proof.

Let $x_1, \ldots, x_{\alpha} \in V$ be a stable set in~$G$ of maximal cardinality. 
Then the measure
\begin{equation*}
\frac{1}{\alpha} \left(\sum_{i=1}^{\alpha} \delta_{x_i}\right) \otimes \left(\sum_{i=1}^{\alpha} \delta_{x_i}\right)
\end{equation*}
is a feasible solution of~\eqref{eq:completelypositive} with objective
value $\alpha$. Hence, $\lambda \geq \alpha$.

In order to prove the reverse inequality we first show that set
$\mathcal{F}_D$ of feasible solutions of~\eqref{eq:completelypositive}
is weak* compact. For this define
\begin{equation*}
\mathcal{F} = \{t(\mu_0 + \mu_1) : (\mu_0,\mu_1) \in S_1, t \in [1, \alpha]\},
\end{equation*}
where
\begin{equation*}
S_1 = \{(\mu_0, \mu_1) \in M(D) \times M(\overline{E}) : \mu_0 + \mu_1 \in \mathcal{CP}_V, \mu_0(D) + \mu_1(\overline{E}) \leq 1\}.
\end{equation*}
By Theorem of Banach-Alaoglu, the set $S_1$ is weak* compact, so
$\mathcal{F}$ is weak* compact as well.

Consider the convex cone
\begin{equation*}
\mathcal{M}_G = \{\mu \in \mathcal{CP}_V : \supp \mu \subseteq D \cup \overline{E}\}.
\end{equation*}
It follows from Theorem~\ref{th:extreme rays} that the extreme rays
of~$\mathcal{M}_G$ are product measures $\rho \otimes \rho$. Furthermore, since $G$ is a
topological packing graph, the
extreme rays of~$\mathcal{M}_G$ have to be of the form
\begin{equation}
\label{eq:feasiblerays}
\left(\sum_{i=1}^N a_i \delta_{x_i}\right) \otimes \left(\sum_{i=1}^N
  a_i \delta_{x_i}\right) \quad  \text{with} \quad a_i \geq 0,\; x_1, \ldots, x_N \text{ a stable set of $G$,}
\end{equation}
because the restriction of $\rho \otimes \rho$ to $D$ has finite support since for every point
$x \in V$ there is an open neighborhood $U$ of $x$ with
$(\rho \otimes \rho) (U \times U) \cap D = (\rho \otimes \rho) (\{(x,x)\})$.

Now let
\[
\mu = \left(\sum_{i=1}^N a_i \delta_{x_i}\right) \otimes \left(\sum_{i=1}^N
  a_i \delta_{x_i}\right) \in \mathcal{F}_D
\]
be a feasible solution of~\eqref{eq:completelypositive} which lies in
an extreme ray of~$\mathcal{M}_G$. We have
\[
\mu(V \times V) = \left(\sum_{i=1}^N a_i\right)^2 \quad \text{and} \quad
\mu(D) = \sum_{i=1}^N a_i^2 = 1.
\]
Write $\mu = \nu_0 + \nu_1$ with $\nu_0
\in M(D)$ and $\nu_1 \in M(\overline{E})$. Let $s$ be a real number
such that $s(\nu_0 + \nu_1)(V \times V) = 1$. Then setting $\mu_0 =
s\nu_0$, $\mu_1 = s \nu_1$, $t = \tfrac{1}{s}$, shows that $s \in [\tfrac{1}{\alpha}, 1]$ 
because of the  Cauchy-Schwartz inequality
\[
1 \geq s = \frac{1}{\left(\sum_{i=1}^N a_i\right)^2} \geq \frac{1}{N
  \sum_{i=1}^N a_i^2} = \frac{1}{N} \geq \frac{1}{\alpha}.
\]
Hence $\mu \in \mathcal{F}$, and consequently $\mathcal{F}_D \subseteq
\mathcal{F}$. This shows that $\mathcal{F}_D$  is weak* compact
because $\mathcal{F}_D$ is closed.

Because of this compactness, the supremum
of~\eqref{eq:completelypositive} is attained at an extreme point of $\mathcal{F}_D$.
Suppose $(\sum_{i=1}^N a_i
\delta_{x_i}) \otimes (\sum_{i=1}^N a_i
\delta_{x_i})$ is a maximizer of~\eqref{eq:completelypositive}. Then
again by Cauchy-Schwartz we get that
\begin{equation*}
\lambda = \left(\sum_{i=1}^N  a_i\right)^2 \leq N \sum_{i=1}^N a_i^2 =
N \leq \alpha,
\end{equation*}
and the claim of the theorem follows.
\end{proof}

\subsection{Copositive formulation}
\label{ssec:copositive}

In this section, we prove our main result, Theorem~\ref{th:main}, by showing that we have 
strong duality between~\eqref{eq:copositive} and~\eqref{eq:completelypositive}.

\begin{theorem}
\label{th:nodualitygap}
There is no duality gap between the primal copositive
program~\eqref{eq:copositive} and the dual completely positive
program~\eqref{eq:completelypositive}. In particular, the optimal
value of both programs equals~$\alpha(G)$.
\end{theorem}

For the proof of this theorem we make use of a variant of the zero duality gap
theorem of a primal-dual pair of conic linear programs, see Barvinok
\cite[Chapter IV.7.2]{Barvinok}:
By dualizing the statement of \cite[Problem 3 in Chapter
IV.7.2]{Barvinok} we see that if the cone
\[
\{(d - A^* y, \langle b, y \rangle_2) : y \in M(V) \times
M(\overline{E}), d \in \mathbb{R}_{\geq 0} \times \mathcal{CP}_V\}
\]
is closed in $\mathbb{R} \times M(V \times V)_{\sym} \times \mathbb{R}$, then
there is no duality gap.

\medskip

To show this closedness condition we need again Lemma~\ref{l1} and the
following lemma which is a slight modification of
\cite[Lemma~IV.7.3]{Barvinok}.

\begin{lemma}\label{l2}
  Let $V$ and $W$ be topological vector spaces, let
  $\mathcal{K} \subseteq V$ be a cone such that there is a compact set
  $\mathcal{B} \subseteq V$ with $0\not\in \mathcal{B}$ and
  $\mathcal{K} = \bigcup_{\lambda \geq 0} \lambda \mathcal{B}$.  Let
  $T \colon V \to W$ be a continuous linear transformation such that
  $\ker T \cap \mathcal{K} = \{0\}$. Then $T(\mathcal{K}) \subseteq W$ is a closed convex
  cone.
\end{lemma}

\begin{proof}
  Obviously $T(\mathcal{K})$ is a convex cone.  The set
  $\mathcal{B}' = T(\mathcal{B})$ is compact, $0 \notin \mathcal{B}'$,
  and
  $T(\mathcal{K}) = \bigcup_{\lambda \geq 0} \lambda
  \mathcal{B}'$.
  Applying Lemma \ref{l1} gives that $T(\mathcal{K})$ is closed.
\end{proof}

Now we are ready for the proof of the theorem:

\begin{proof}[Proof of Theorem~\ref{th:nodualitygap}]
Consider the continuous linear
transformation
\[
T(d, y) = (d - A^* y, \langle b, y \rangle_2).
\]
We have already seen that the cone has a compact base. Suppose $(d,
y)$ lie in the kernel of~$T$. Then the condition $\langle b, y \rangle_2 = 0$
forces $y$ to be zero. This forces $d = 0$ and we can apply
Lemma~\ref{l2} to complete the proof of the theorem.
\end{proof}

\subsection{Copositive formulation for the weighted stability number}
\label{ssec:weighted}

In some situations one wishes to consider packing problems with
different types of objects, having different sizes; for instance the
problem of packing spherical caps having different radii as considered
by de~Laat, Oliveira, and Vallentin~\cite{LaatOliveiraVallentin}. In
these cases it is helpful to use a weighted version of the copositive
problem formulation which is presented in the next theorem. We omit
its proof here since it is completely analogous to the one of
Theorems~\ref{th:nodualitygap} and~\ref{th:completelypositive}.  The
only difference is that we are now given a continuous weight function
$w : V \to \mathbb{R}_{\geq 0}$ for the vertex set, and in our
optimization problems we replace the objective function
\[
\mu(V \times V) =
\int_V \int_V \, d\mu(x,y)
\qquad \text{by} \qquad
\int_V \int_V \sqrt{w(x)w(y)} \, d\mu(x,y).
\]

\begin{theorem}
  Let $G = (V, E)$ be a compact topological packing graph and let
  $w : V \to \mathbb{R}_{\geq 0}$ be a continuous weight function for
  the vertex set. Then the weighted stability number $\alpha_w(G)$
  defined by
\[
\alpha_w(G) = \max\left\{\sum_{x \in S} w(x) : \text{$S$ stable set of $G$}\right\}
\]
has the following copositive formulation
\begin{equation} \label{eq:mwcp}
\begin{array}{rll}
\alpha_w(G) = \inf & t\\
&  t \in \R,\; K \in \mathcal{COP}_V \\
& K(x,x) = t-w(x) & \text{for all $x \in V$} \\
& K(x,y) = -\sqrt{w(x)w(y)} & \text{for all $\{x,y\} \not\in E$.}
\end{array}
\end{equation}
\end{theorem}
For the finite case, Bomze~\cite{Bomze1998} showed that the maximum
weight clique problem can be formulated as a standard quadratic
problem. With the techniques
from~\cite{BomzeDuerDeKlerkQuistRoosTerlaky} this in turn can be
written as a copositive problem of which~\eqref{eq:mwcp} is the
infinite counterpart.

\section{Copositive formulation of the kissing number}
\label{sec:kissing}

In this section we give a copositive formulation of the kissing number problem. We
show that in this case the copositive program can be equivalently
transformed into a semi-infinite linear program. We start with
the original copositive formulation:
\begin{equation*}
\begin{array}{rll}
\inf & t\\
& t \in \mathbb{R}, \; K \in \mathcal{COP}_{S^{n-1}} \\
& K(x,x) = t - 1 & \text{for all $x \in S^{n-1}$} \\
& K(x,y) = -1 & \text{for all $x, y \in S^{n-1}$ with $x^{\sf T} y \in [-1, 1/2]$.}
\end{array}
\end{equation*}
Since the packing graph is invariant under the orthogonal group, also
the copositive formulation is invariant under this group.  By
convexity we can restrict the copositive formulation above to
copositive kernels which are invariant under the orthogonal group. So
$K(x,y)$ only depends on the inner product $x^{\sf T} y$. 

By Stone-Weierstrass we know that polynomials lie dense in
$C([-1,1])$, so we approximate $K(x,y)$ by
$\sum_{k=0}^d c_k (x^{\sf T} y)^k$. Then by Lemma~\ref{lem:pfender},
the copositivity condition $K \in \mathcal{COP}_{S^{n-1}}$ translates
to
\begin{equation}
\label{eq:copositiveconstraint}
  \sum_{i=1}^N \sum_{j=1}^N \sum_{k = 0}^d c_k (x_i^{\sf T} x_j)^k 
  \geq 0 \quad \text{for all $N \in \mathbb{N}$ and $x_1, \ldots, x_N \in S^{n-1}$}.
\end{equation}
The other constraints of the above copositive problem translate
likewise, and observing that $x^{\sf T} x = 1$ for $x \in S^{n-1}$, we
get the following semi-infinite linear program whose optimal value
converges to the kissing number if the degree $d$ tends to infinity:
\begin{equation*}
\begin{array}{rll}
\inf & 1 + \sum_{k = 0}^d c_k\\
& c_0, \ldots, c_d \in \R\\[.5ex]
& \sum\limits_{i=1}^N \sum\limits_{j=1}^N \sum\limits_{k = 0}^d c_k
  (x_i^{\sf T} x_j)^k \geq
0 & \text{for all $N \in \mathbb{N}$ and $x_1, \ldots, x_N \in S^{n-1}$}\\[.5ex]
& \sum\limits_{k = 0}^d c_k s^k \leq -1 & \text{for all $s \in [-1,1/2]$.}
\end{array}
\end{equation*}
We impose the condition $\sum_{k = 0}^d c_k s^k \leq -1$ instead of
$\sum_{k=0}^d c_k s^k = -1$ to make the problem feasible for finite
degree~$d$. It is easy to see that this relaxation is not effecting
the optimal value when $d$ tends to infinity.

Note here that all the difficulty of the problem lies in the
copositivity constraint~\eqref {eq:copositiveconstraint}.  In contrast
to this, the other constraint $\sum_{k=0}^d c_k s^k \leq -1$ for all
$s \in [-1,1/2]$ is computationally relatively easy. Although it gives
infinitely many linear conditions on the coefficients~$c_k$, it can be
modeled equivalently as a semidefinite constraint using the sums of
squares techniques for polynomial optimization; see for instance
Parrilo~\cite{Parrilo} and Lasserre~\cite{Lasserre}.

If, instead of requiring copositivity of the invariant kernel
\[
(x,y) \mapsto \sum_{k=0}^d c_k (x^{\sf T} y)^k
\]
we impose the weaker constraint that this kernel should be positive
semidefinite, then things become considerably simpler. Using harmonic
analysis on the unit sphere, by Schoenberg's
theorem~\cite{Schoenberg}, one can identify this class of kernels
explicitly, namely these are the kernels which can be written as
\[
(x,y) \mapsto \sum_{k=0}^d g_k P_k^{((n-3)/2,(n-3)/2)}(x^{\sf T} y)
\quad \text{with} \quad g_0, \ldots, g_d \geq 0,
\]
where $P_k^{((n-3)/2,(n-3)/2)}$ is the Jacobi polynomial of degree~$k$
with parameters $((n-3)/2, (n-3)/2)$. Thus requiring this weaker
constraint instead of the copositivity constraints yields the linear
programming bound for the kissing number due to Delsarte, Goethals,
and Seidel~\cite{DelsarteGoethalsSeidel1977}. This bound is known to
be tight in a few cases only, namely for $n = 1, 2, 8, 24$.

\section{Future work}

In this paper we gave a copositive formulation of the stability number
of compact topological packing graphs. This condition guarantees in
particular that all stable sets are finite. Sometimes one is also
interested in stable sets which are infinite, measurable sets. For
instance, what is the measurable stability number of the graph on the
unit sphere where two vertices are adjacent whenever they are
orthogonal?  Semidefinite relaxations for problems of this kind have
been proposed by Bachoc, Nebe, Oliveira, and Vallentin
\cite{BachocNebeOliveiraVallentin}. However, the bound which one can
obtain by this method is very weak for the orthogonality graph on the
unit sphere, and it is difficult to find additional valid inequalities
to improve it significantly (see DeCorte and Pikhurko
\cite{DeCorte2015a} for the case $n = 3$).  For this reason we think
that it would be interesting to derive a copositive formulation for
this problem to be able to derive stronger bounds.

\section*{Acknowledgements}

We would like to thank Christine Bachoc and Evan DeCorte for very helpful
discussions. We would also like to thank the anonymous referees whose
comments helped to improve the presentation of the paper.

\end{document}